\documentclass[11pt]{amsart}
\usepackage{rlepsf, graphicx, amssymb, %epstopdf, 
amsmath, amssymb}%, draftcopy}

\theoremstyle{plain}
\newtheorem{thm}{Theorem}[section]
\newtheorem{lem}[thm]{Lemma}
\newtheorem{cor}[thm]{Corollary}

\theoremstyle{definition}
\newtheorem*{Def}{Definition}

\newtheorem{rem}[thm]{Remark}

\theoremstyle{remark}

\def\C{\mathbb{C}}
\def\R{\mathbb{R}}
\def\Z{\mathbb{Z}}

\def\dfn#1{{\em #1}}

\title[Fibered Transverse Knots]{Fibered Transverse Knots and the Bennequin Bound}
\author{John B.\ Etnyre}
\address{
    School of Mathematics,
    Georgia Institute of Technology,
    686 Cherry St.,
    Atlanta, GA  30332-0160}
\email{etnyre@math.gatech.edu}
\urladdr{http://math.gatech.edu/\char126 etnyre}
\author{Jeremy Van Horn-Morris}
\address{CIRGET, Universit\'e du Qu\'ebec \`a Montr\'eal, Case Postale 8888, succursale Centre-Ville, Montr\'eal (Qu\'ebec) H3C 3P8, Canada} 
\email{jvanhorn@math.utexas.edu}
\urladdr{}

\begin{document}
\begin{abstract} 
We prove that a nicely fibered link (by which we mean the binding of an open book) in a tight contact manifold $(M,\xi)$ with zero Giroux torsion has a transverse representative realizing the  Bennequin bound if and only if the contact structure it supports (since it is also the binding of an open book) is $\xi.$ This gives a geometric reason for the non-sharpness of the Bennequin bound for fibered links. We also give the classification, up to contactomorphism, of maximal self-linking number links in these knot types. Moreover, in the standard tight contact structure on $S^3$ we classify, up to transverse isotopy, transverse knots with maximal self-linking number in the knots types given as closures of positive braids and given as fibered strongly quasi-positive knots. We derive several braid theoretic corollaries from this. In particular, we give conditions under which quasi-positive braids are related by positive Markov stabilizations and when a minimal braid index representative of a knot is quasi-positive. Finally, we observe that our main result can be used to show, and make rigorous the statement, that contact structures on a given manifold are in a strong sense classified by the transverse knot theory they support. 
\end{abstract}
\maketitle

%%%%%%%%%%%%%%%%%%%%%%%%%%%%%%%%%%%%%%%%%%%%%%%%%%%%%%%%
\section{Introduction}
%%%%%%%%%%%%%%%%%%%%%%%%%%%%%%%%%%%%%%%%%%%%%%%%%%%%%%%%
In \cite{Bennequin83}, Bennequin established the following inequality for knots $K$ transverse to the standard tight contact structure on $S^3$
\[
sl(K)\leq -\chi(\Sigma),
\]
where $\Sigma$ is any Seifert surface for $K$ and $sl(K)$ denotes the self-linking number of $K.$ Eliashberg proved this inequality for transverse knots and links in any tight contact structure \cite{Eliashberg92a}. Since that time there has been much work finding other bounds for transverse knots in $S^3$ \cite{Etnyre05,FuchsTabachnikov97,Ferrand02} and some work finding other bounds in general manifolds \cite{Hedden??, LiscaMatic96}.  For some links, other bounds are more restrictive and the Bennequin inequality is not sharp.  In general there is little geometric understanding of the knot types for which the Bennequin inequality is sharp. Building on an interesting result of Hedden \cite{Hedden??} discussed below, our first goal in this paper is to rectify this situation for fibered links by giving a specific necessary and sufficient condition for sharpness. Once we have done this we will discuss several corollaries concerning the classification of transversal knots, braids representing quasi-positive links and the determination of a contact structure via its transverse knot theory.

\subsection{Sharpness in the Bennequin inequality}
We say an oriented link $L$ in a 3-manifold $M$ is fibered if there is a fibration $\pi:(M\setminus L)\to S^1$ such that the closure of any fiber is an embedded surface whose boundary, as an oriented manifold, is equal to $L.$ This is a slightly more restrictive definition of a fibration in a general 3-manifolds than is usually given and, of course, is equivalent to $(L,\pi)$ being an open book decomposition for $M$.  As there might be many different fibrations with binding $L$, we use the notation $(L, \Sigma)$ to emphasize a particular fiber surface $\Sigma$.  Given an open book decomposition $(L, \Sigma)$, Thurston and Winkelnkemper constructed in \cite{ThurstonWinkelnkemper75} a contact structure $\xi_{(L,\Sigma)}$ that is supported by the fibration.  Examining the construction, one easily sees that $L$ is naturally a transverse link in $\xi_{(L,\Sigma)}$ and the self-linking number of $(L, \Sigma)$ is $-\chi(\Sigma)$. Thus we see that fibered links in a tight contact manifold that support the given contact structure have transverse representatives realizing equality in the Bennequin inequality. More generally, we say that the fibered link $(L, \Sigma)$ in a contact manifold $(M,\xi)$ realizes the Bennequin bound if there is a transverse realization of $L$ such that  $sl_\xi{(L,\Sigma)} = -\chi(\Sigma)$, noting that if $H_2(M; \Z)$ is non-trivial, then the self-linking number may depend on the surface $\Sigma.$  In \cite{Hedden08}, Hedden showed that in the standard tight contact structure on $S^3$ the Bennequin bound is sharp for a fibered knot if and only if the contact structure defined by the knot is tight. Hedden used a Heegaard-Floer invariant of knots to establish this result and proves a similar result for contact manifolds with non-vanishing contact invariant (thought, unlike in the case of $S^3,$ the exact contact structure supported by the knot cannot necessarily be determined). In general, one would not expect the simple algebraic condition of equality in the Bennequin inequality for a transverse representative of a link could have the strong geometric consequence that the link is the binding of an open book decomposition supporting the ambient contact structure, however we show that this is essentially the case.

\begin{thm}\label{mainknot}
Let $M$ be a closed 3-manifold and $\xi$ a tight contact structure on $M.$  A fibered link $(L, \Sigma)$ realizes the Bennequin bound in $(M,\xi)$ if and only if $\xi$ is supported by the open book $(L, \Sigma)$ or is obtained from $\xi_{(L,\Sigma)}$ by adding Giroux torsion.
\end{thm}
%We note that when $M=S^3$ this result is due to Hedden \cite{Hedden??, Hedden08}
We have the following immediate corollaries.
\begin{cor}\label{cor:atoroidal}
If $\xi$ is a tight contact structure on $M$ with zero Giroux torsion (for example if $M$ is atoroidal), then a fibered link $(L, \Sigma)$ realizes the Bennequin bound in $(M,\xi)$ if and only if $\xi$ is supported by $(L, \Sigma)$.
\end{cor}
\begin{cor}\label{paclass}
If $\xi$ is a tight contact structure on $M$, then a fibered knot with pseudo-Anosov monodromy $(L,\Sigma)$ realizes the Bennequin bound in $(M,\xi)$ if and only if $\xi$ is supported by $(L,\Sigma).$
\end{cor}

Given an oriented nicely fibered link $L$ in a rational homology sphere, there is a natural homotopy class of plane field associated to $L.$  Namely, take the homotopy class of the contact structure associated to the open book with binding $L$.  Another immediate corollary of our main theorem is the following which also follows using Heegaard-Floer theory as shown in \cite{Hedden08}.
\begin{cor}
Let $M$ be a rational homology sphere and $L$ an oriented fibered link in $M.$
If the homotopy class of plane field associated to $L$ is not the same as the homotopy class of a given tight contact structure $\xi$, the Bennequin inequality is a strict inequality. 
\end{cor}

It would be interesting to better understand if the difference between $-\chi(\Sigma)$ and the maximal self-linking number of a fibered knot $K$ in a tight contact structure $\xi$ was somehow related to the difference between the Euler class and Hopf invariant of $\xi$ and the plane field associated to $K.$

Since one implication in Theorem~\ref{mainknot} is obvious from the discussion above, the theorem is a consequence of the following.

\begin{thm}  \label{thm:G-compat}
Let $(L, \Sigma)$ be a fibered transverse link in a contact 3-manifold $(M, \xi)$ and assume that $\xi$ is tight when restricted to $M\setminus L$.  If $sl_\xi{(L,\Sigma)} = -\chi(\Sigma)$, then either
\begin{enumerate}
\item $\xi$ is supported by $(L, \Sigma)$ or
\item $\xi$ is obtained from $\xi_{(L,\Sigma)}$ by adding Giroux torsion along tori which are incompressible in the complement of $L$.  
\end{enumerate}
\end{thm}

The main new technical idea used in the proofs of our theorems is the idea of a contact structure being quasi-compatible with an open book. That is a contact structure $\xi$ is quasi-compatible with an open book $(L,\Sigma)$ if there is a contact vector field positively transverse to the pages and positively tangent to the boundary of the open book. If the contact vector field is a Reeb vector field for $\xi$ then the open book is actually compatible. We expect further study of this generalization of compatibility to be useful in understanding questions about Giroux torsion, and hope to return to this in a future paper. Several of the arguments in the proofs our theorems can be thought of as extending standard convex surface theory to convex surfaces with transverse boundary. As is well known to the experts, many of the most useful theorems do not carry over to this situation; however, we show that some of the theorems do carry over, though sometimes in a somewhat restricted form. 

\subsection{The classification of transverse knots with maximal possible self-linking number}
Not only do we understand when a fibered link realizes the Bennequin inequality, we can also classify maximal self-linking number transverse links when this bound is achieved.  We say a transverse link $L$ has a \emph{maximal self-linking number} if there is some Seifert surface $\Sigma$ for $L$ such that $sl_{\Sigma}(L)=-\chi(\Sigma).$
%if it is equal to the Thurston norm for some nontrivial subset of $H_2(M-\nu L,\partial (M-\nu L))$.  Two transverse links $L_1$ and $L_2$ have the same self-linking number if $sl_\xi(L_1)$ and $sl_\xi(L_2)$ represent the same element in $Hom (H_2(M-\nu L,\partial (M-\nu L)), \mathbb{Z})$.  The following theorem implies that for fibered transverse links realizing the Bennequin bound, the self-linking number is the only transverse invariant.

\begin{thm}\label{knotthm}
Suppose that $L$ is a fibered link in $S^3$ that supports the standard tight contact structure. There is a unique transverse link, up to transverse isotopy, in the topological knot type of $L$ with self-linking number $-\chi(L)$, where $\chi(L)$ denotes the maximal Euler characteristic of surfaces with boundary $L.$ 

Moreover, suppose $L$ is a fibered link in a 3-manifold $M$ and the fibration $(L,\pi,\Sigma)$ supports a tight contact structure with Giroux torsion equal to zero. Any two transverse representatives of $L$ with $sl=-\chi(\Sigma)$ are contactomorphic. 
\end{thm}

\begin{rem} It may be of interest to note that the contactomorphism in Theorem~\ref{knotthm} is topologically trivial and induced by an isotopy.  For a pair of maximal fibered links, $(L_0, \Sigma)$ and $(L_1, \Sigma)$, there is an isotopy of pairs $(\xi_t, L_t)$, for $0 \leq t \leq 1$, where $\xi_t$ is a contact structure, $L_t$ is a transverse link in $\xi_t$, and $\xi_0 = \xi = \xi_1$.  In other words, once we allow the contact structure to deform, the two links are transversely isotopic.  In particular, such pairs would be all almost indistinguishable by a transverse link invariant.
\end{rem}

\begin{rem}
Notice that if $L$ is a fibered link in a homology sphere $M$ and $(L, \Sigma)$ supports a tight contact structure with Giroux torsion equal to zero, then the self-linking numbers of the components of a transverse representative of $L$ with maximal self-linking number are fixed. For a general link, one can have several different transverse realizations of the link with maximal total self-linking number, but with different self-linking numbers on the components. The last theorem says this cannot happen in the situation we are considering. 
\end{rem}

\begin{rem}
For overtwisted manifolds, there is no bound on the Giroux torsion.  In this case, there is a statement completely contrary to the uniqueness of the Theorem~\ref{knotthm}. In forthcoming work it will be shown that if 
%\begin{thm} \label{thm:OT}
$(L, \Sigma)$ supports an overtwisted contact structure, $\xi$, then there are infinitely many non-isotopic transverse knots with self-linking number $-\chi(\Sigma)$ in $\xi$, each having a distinct contact structure on the complement.
%\end{thm}
\end{rem}

One might hope that knot types of fibered knots in $S^3$ that support the standard tight contact structure (or perhaps any a contact manifold with zero Giroux torsion) might be transversely simple (that is transverse knots in this knot type are determined by their self-linking number).  This is not the case. For example, the $(2,3)$-cable of the $(2,3)$-torus knot is a fibered knot that supports the tight contact structure on $S^3,$ but in \cite{EtnyreHonda05} it was shown not to be transversely simple. Nonetheless, Theorem~\ref{knotthm} gives a starting point for the classification of such knots. According to the strategy laid out in \cite{Etnyre99}, to classify knots in the knot types under consideration one would only need to show that any transverse knot in the topological knot type with non-maximal self-linking number can be destabilized. While this is the case for some such knots (e.g., positive torus knots), it is not true for others (e.g., the $(2,3)$-cable of the $(2,3)$-torus knot). Thus, to give a classification for a given knot type one needs to understand the destabilization issue: determine those transverse knots that destabilize and classify those that do not. In any event, Theorem~\ref{knotthm} gives a rather good start to the classification of transverse  knots in several cases. 

\subsection{Braid theory and quasi-positivity}
We note the following simple corollary of Theorem~\ref{knotthm}.

\begin{cor}\label{classifymax}
Let $L_B$ be the closure of a positive braid $B$ in $S^3.$ Then there is a unique, up to transverse isotopy, transverse link in the standard tight contact structure on $S^3$ with the knot type of $L_B$ and self-linking number equal to $-\chi(L)=-n+a,$ where $n$ is the braid index of $B$ and $a$ is the length of the word expressing $B$ in terms of the standard generators of the braid group.  

More generally, if $L$ is a fibered strongly quasi-positive link in $S^3$ bounding the quasi-positive Seifert surface $\Sigma$ then there is a unique, up to transverse isotopy, transverse link in the standard tight contact structure on $S^3$ with the knot type of $L$ and self-linking number equal to $-\chi(\Sigma).$
\end{cor}
Recall a link is called \dfn{strongly quasi-positive} if it can be represented in some braid group $B_n$ as a word of the form $\Pi_{k=1}^m \sigma_{i_k,j_k}$ where $\sigma_1,\ldots, \sigma_{n-1}$ are the standard generators of the braid group and 
\[
\sigma_{i,j}= (\sigma_i\ldots \sigma_{j-2})\sigma_{j-1}(\sigma_i\ldots\sigma_{j-2})^{-1}.
\]
For a thorough discussion of strongly quasi-positive links we refer to \cite{Rudolph05}, but we note that positive knots --- that is knots with projections having only positive crossings --- are strongly quasi-positive. So this corollary applies to fibered positive knots. While the appearance of strongly quasi-positive links may appear mysterious in this corollary, it is really quite natural given Corollary~\ref{cor:atoroidal} above and Hedden's observation \cite{Hedden??} that that a fibered link in $S^3$ supports the standard tight contact structure on $S^3$ if and only if it is strongly quasi-positive. 

We now explore some consequence of our theorems above to problems in braid theory. Specifically to problems concerning positive braids and more generally quasi-positive braids. Recall a braid is \dfn{quasi-positive} if is a product of conjugates of the standard generators of the braid group and a link is quasi-positive if it can be represented by a quasi-positive braid. Clearly strongly quasi-positive links are quasi-positive. By way of motivation for this definition we note that quasi-positive links are precisely those that are formed by transversely intersecting the unit 3-sphere in $\C^2$ with a complex analytic plane curve \cite{BoileauOrevkov01, Rudolph83}. Moreover Orevkov has an approach to Hilbert's sixteenth problem concerning real algebraic curves using quasi-positivity \cite{Orevkov99}. In \cite{Orevkov00} two fundamental questions concerning quasi-positive links were asked. They were

\smallskip
\begin{minipage}[h]{5.5in}
\em Given two quasi-positive braids representing a fixed link, are they related by positive Markov moves (and braid conjugation)?
\end{minipage}

\smallskip
\noindent
and
\smallskip

\begin{minipage}[h]{5.5in}
\em Given a quasi-positive link, is any minimal braid index braid representative of the link quasi-positive?
\end{minipage}

\smallskip
\noindent
We have some partial answers below, but first we need an observation. It is well know, see \cite{Bennequin83}, that the closure of a braid $B$ gives a transverse link in the standard tight contact structure on $S^3$ and moreover its self-linking number is $a(B)-n(b)$ where $a(B)$ is the algebraic length of $B$ and $n(B)$ is the braid index of $B.$ Thus the Bennequin inequality can be written
\[
a(B)\leq n(B)-\chi(L_B)
\] 
where $L_B$ is the link obtained as the closure of $B$ and $\chi(L_B)$ is maximal Euler characteristic of a Seifert surface for $L.$ With this in mind we have the following criterion for a braid to be quasi-positive. 

\begin{thm}\label{qp-rep}
Let $L$ be a fibered strongly quasi-positive link and $B$ any braid representing $L.$ Then $a(B)=n(B)-\chi(L)$ if and only if $B$ is quasi-positive.
\end{thm}

Using this theorem we can show the following.
\begin{thm}\label{stab}
Let $L$ be a fibered strongly quasi-positive link. Any two quasi-positive braids representing $L$ are related by positive Markov moves and braid conjugation. 
\end{thm}
We point out an interesting special case of this theorem in the following immediate corollary. 
\begin{cor}
Any two positive braids representing a link $L$ are related by positive Markov moves and braid isotopy.\qed
\end{cor}

To state our result for minimal braid index representatives of (strongly) quasi-positive knots we recall a conjecture of K.~Kawamuro, \cite{Kawamuro06}. Given a link type $L$ with braid index $b>0$ is there an integer $w$ such that any braid $B$ whose closure is in the link type $L$ satisfies 
\[
b+|a(B)-w|\leq n(B).
\]
We call this the Kawamuro Braid Geography Conjecture. One easily sees that this generalizes the Jones conjecture that there is a unique value for $a(B)$ among braids with braid index $b$ realizing the link type $L.$ Moreover, one may check that this conjecture is true for any knot type for which the Morton-Franks-Williams (MFW) inequality is sharp. Recall the MFW-inequality says that for a link type $L$ we have 
\[a({B})-n({B})+1\leq d_-(L)\leq d_+(L)\leq a({B})+n({B})-1\]
where $B$ is any braid whose closure is in the link type $L$ and $d_+(L)$ ($d_-(L)$) is the maximal (minimal) degree in the $a$-variable of the two variable Jones polynomial. One says the inequality is sharp if there is a fixed braid for which the first and last inequalities are equalities. Sharpness is known for closures of positive braids that contain a full right handed twist \cite{FranksWilliams87} and for alternating fibered links \cite{Murasugi91}. It is known that the MFW-inequality is not always sharp, but it is still unknown if the Braid Geography Conjecture is true or not; however, in \cite{Kawamuro06} it was shown that the class of links for which the Braid Geography Conjecture is true is closed under cabling and connected sums. 
\begin{thm}\label{minrep}
If $L$ is a fibered strongly quasi-positive knot type satisfying the Braid Geography Conjecture then any braid with minimal braid index representing $B$ is quasi-positive. 

In particular, if $L$ is the closure of a positive braid with a full twist (or a sufficiently positive cable or connected sum of such knots), then any braid of minimal index representing $L$ is quasi-positive. 
\end{thm}
%We note that for $L$ as in the end of the theorem any minimal braid index braid representing $L$ can be made into a positive braid by strictly positive Markov stabilizations. It is interesting to ask if a quasi-positive braid becomes positive after a positive stabilization then is it actually positive?

\subsection{Transverse classification of contact structures}
As a last application of our main theorems we note that transverse knots classify contact structures. More specifically if ${K}$ denotes a topological knot type in a manifold $M$ then we denote by $\mathcal{T}^\xi({K})$ the set of knots in $M$ that are isotopic to ${K}$ and transverse to $\xi.$ Moreover we set
\[
\mathcal{T}_n^\xi({K})=\{T\in{T}^\xi({K}) : sl(T)=n \}
\]
for any integer $n.$ Here we could take $\mathcal{T}^\xi({K})$ to be transverse knots up to contactomorphism or up to transverse isotopy, in this paper we take them up to contactomorphism. 
(Notice that if $\mathcal{K}$ bounds non-homologous Seifert surfaces then we should consider pairs of transverse knots and Seifert surfaces.) We say a collection of knots $\mathcal{C}=\{{K}_\alpha\}_{\alpha\in A}$ in $M,$ indexed by some set $A,$ \dfn{transversely classifies contact structures on $M$} up to contactomorphism, respectively isotopy, if two contact structures $\xi_1$ and $\xi_2$ on $M$ are contactomorphic, respectively isotopic, if and only if $|\mathcal{T}_n^{\xi_1}({K}_\alpha)|=|\mathcal{T}_n^{\xi_2}({K}_\alpha)|$ for all $\alpha\in A$ and $n\in \Z.$ 
\begin{thm}\label{classifies}
For a given manifold $M$ the collection 
\[
\mathcal{C}_\text{fpa}=\{\text{fibered knots in $M$ with pseudo-Anosov monodromy}\}
\]
transversely classifies contact structures on $M$ up to isotopy.
\end{thm}
This theorem makes precise the notion that the knot theory a contact structure supports determines the contact structure. One can make similar definitions for a collection of knots to Legendrian classifying contact structures, but it is unclear if a similar theorem holds in this case. It is also an interesting open question to find smaller collections of knots that classify contact structures on certain manifolds. Once can also consider collections of knots that only classify tight contact structures. For example one may easily verify that if $\mathcal{C}$ consists of three linear knots in $T^3$ that span the first homology of $T^3$ then $\mathcal{C}$ Legendrian classifies tight contact structures up to isotopy.

{\em Acknowledgments:} The authors thank Francis Bonahon for a useful conversation and Matthew Hedden for several stimulating conversations that inspired the work in this paper. We also thank Joan Birman and Keiko Kawamuro for several valuable e-mail exchanges. Lastly we thank Bulent Tosun for several enlightening conversations which lead to Theorem~\ref{classifies} above. JE was partially supported by NSF grants DMS-0804820, DMS-0707509 and DMS-0244663. 

%%%%%%%%%%%%%%%%%%%%%%%%%%%%%%%%%%%%%%%%%%%%%%%%%%%%%%%%
\section{Preliminary notions and notations.}\label{sec:prelim}
%%%%%%%%%%%%%%%%%%%%%%%%%%%%%%%%%%%%%%%%%%%%%%%%%%%%%%%%

Throughout the paper, we will use standard results concerning transverse and Legendrian links, characteristic foliations and convex surfaces. Instead of recalling all the relevant, and by now well-known, theorems and definitions here, we refer the reader to the paper Sections~2 and 3 of \cite{EtnyreHonda01b} where all the necessary background can be found. (For a more thorough discussion also see the papers \cite{Etnyre05, Giroux91, Honda00a}.)  We now recall the less standard facts and sketch some of the proofs which are mostly minor modifications of the ones given in the above references.

\begin{lem}\label{dividing}
Let $\Sigma$ be an oriented surface whose boundary is positively transverse to $\xi$. If the characteristic foliation on
$\Sigma$ is Morse-Smale, then there is a contact vector field $v$ that is transverse to $\Sigma$ and
induces dividing curves on $\Sigma$ that are disjoint from $\partial \Sigma.$
\end{lem}
Many results concerning closed convex surfaces and convex surfaces with Legendrian boundary do not hold for
convex surfaces with transverse boundary. Nonetheless, it still makes sense to define a convex
surface with transverse boundary, as we do here, and several well known results have weaker analogs in
this situation. The proof of this lemma is a simple modification of the corresponding proof for
closed surfaces found in \cite{Giroux91}.  We sketch a few key points that we will need below.
\begin{proof}[Sketch of proof]
%After a $C^\infty$-small perturbation of $\Sigma$, we can assume that the characteristic foliation on $\Sigma$ is Morse-Smale. 
We define the surface $\Sigma_+'$ to be a small neighborhood of
the positive elliptic singularities and positive periodic orbits, together with a small neighborhood 
about the stable manifolds of the positive hyperbolic points. Similarly, we can define $\Sigma_-'$ 
using a small neighborhood of the negative elliptic singularities, periodic orbits, and boundary, 
together with neighborhoods of the unstable manifolds of the negative hyperbolic points. It is easy to make 
these surfaces disjoint and have their boundaries transverse to the characteristic foliation. 
The surface $\Sigma\setminus (\Sigma_+'\cup \Sigma_-')$ is a disjoint 
union of annuli. If we choose a simple closed curve in the center of each annulus, the resulting 
multicurve will divide the characteristic foliation in the sense of \cite{Giroux91}. Thus we may
construct a contact vector field transverse to $\Sigma$ with the given dividing set which is 
clearly disjoint from the boundary. 
\end{proof}

We also mention that Giroux's theorem allowing for the realization of characteristic foliations
respecting dividing curves holds in a weaker form for convex surfaces with transverse boundary. In
particular we have the following result.
\begin{thm}\label{realize}
Let $\Sigma$ be an oriented convex surface with boundary positively transverse to $\xi$ whose dividing set $\Gamma$ is
disjoint from the boundary. Let $\mathcal{F}$ be a singular foliation that agrees with the
characteristic foliation of $\Sigma$ on the components of $\Sigma\setminus \Gamma$ containing $\partial \Sigma$ and is divided by $\Gamma$.  There is a smooth isotopy $\Sigma_t,$ fixed near the boundary of $\Sigma,$ through convex surfaces with transverse boundary, such that the characteristic foliation on $\Sigma_1$ is $\mathcal{F}$.\qed
\end{thm}
Here again, the proof of this theorem is an obvious generalization of the one given in \cite{Giroux91} for
closed surfaces.

\begin{lem}\label{sl-formula}
Suppose $\Sigma$ is an oriented convex surface with boundary positively transverse to $\xi$ whose dividing curves
$\Gamma$ are disjoint from the boundary. Then
\[sl(\partial \Sigma)=-\bigl(\chi(\Sigma_+)-\chi(\Sigma_-)\bigr),\]
where $\Sigma_\pm$ is the union of components of $\Sigma\setminus \Gamma$ with $\pm$-divergence.
\end{lem}
\begin{proof}
We can assume that the characteristic foliation on $\Sigma$ is Morse-Smale. Recall from
\cite{Bennequin83} that
\begin{equation}\label{slformula}
sl(\partial\Sigma)=-\bigl((e_+-h_+)-(e_--h_-)\bigr),
\end{equation}
where $e_\pm$ is the number of positive/negative elliptic points in the characteristic foliation of
$\Sigma$, and $h_\pm$ is the number of positive/negative hyperbolic points. Further recall that $\Sigma_+$ is
homotopic to the surface obtained by taking the union of a small neighborhood of the positive elliptic points and repelling
periodic orbits, and a small neighborhood of the unstable manifolds
of the positive hyperbolic points. One may now easily check that $\chi(\Sigma_+)=e_+-h_+$ and, similarly, $\chi(\Sigma_-)=e_--h_-.$ 
\end{proof}

\begin{lem}[Torisu 2000, \cite{Torisu00}]\label{support}
The contact structure $\xi$ on $M$ is supported by the open book decomposition $(B,\pi)$ if and only
if for every two pages of the open book that form a smooth surface $\Sigma'$, the contact structure
can be isotoped so that $\Sigma'$ is convex with dividing set $B\subset \Sigma'$ and $\xi$ is tight
when restricted to the components of $M\setminus \Sigma'.$
\end{lem}
\begin{proof}[Sketch of Proof]
It is well known that $\Sigma'$ is a Heegaard surface for $M.$  This can be seen by noticing that if
$P=\Sigma\times[-\epsilon,\epsilon]$ is one of the components of $M\setminus \Sigma'$, and 
$c_1,\ldots, c_k$ is a collection of curves that cut $\Sigma$ into a disk, then $D_1, \ldots, D_k$
is a collection of disks in $P$ that cut $P$ into a ball, where $D_i=c_i\times[-\epsilon,\epsilon].$
In addition, $S_i=(\partial D_i\setminus c_i\times\{-\epsilon\})\cup (\phi(c_i)\times\{-\epsilon\})$
bounds a disk $D'_i$ in $M\setminus P,$ where $\phi$ is the monodromy of the fibration of $M\setminus B.$ 
All the disks $D'_i$ are disjoint and they cut $M\setminus P$ into
a ball.  The two components of $M\setminus \Sigma'$ are thus handlebodies.

One easily checks that under the hypothesis of the lemma, the simple closed curves $\partial D_i$ and
$\partial D_i'$ each intersect the dividing set of $\Sigma'$ twice. Thus, we may Legendrian realize
these curves and make the disks convex. Moreover, as each disk has a single dividing curve 
there is a unique configuration for the dividing
set on these disks. Since contact structure is determined in a neighborhood of $\Sigma'\cup(\cup D_i)\cup (\cup D_i')$ by the dividing set, the handlebodies cut along these disks are 3-balls, and we are assuming the
contact structure on the handlebodies is tight, the contact structure on the handlebodies is unique.
There can thus be only one tight contact structure satisfying the hypothesis of the theorem.  One easily checks that the contact structure supported by the open book decomposition satisfies the hypothesis.
\end{proof}

\begin{lem}[Honda, Kazez and Matic 2005, \cite{HondaKazezMatic07}]\label{moveleft}
Suppose $\Sigma$ is a convex surface containing a disk $D$ such that $D\cap \Gamma_\Sigma$ is as
shown in Figure~\ref{fig:moveright}. Also suppose $\delta$ and $\delta'$ are as shown in the figure.
If there is a bypass for $\Sigma$ attached along $\delta$ from the front, then there is a bypass for
$\Sigma$ attached along $\delta'$ from the front.
\end{lem}
\begin{figure}[ht]
  \relabelbox \small {\epsfysize=1.1truein\centerline {\epsfbox{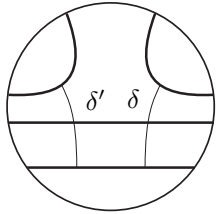}}} \relabel
        {d}{$\delta$} \relabel {2}{$\delta'$} \endrelabelbox
        \caption{If there is a bypass for $\delta$ then there is one for $\delta'$ too.}
        \label{fig:moveright}
\end{figure}

We briefly recall that for a torus $T$, embedded in a contact manifold $(M,\xi)$, the {\em
Giroux torsion} of $T$ is the largest integer $n$ such that there is a contact embedding of $T^2\times[0,1]$
with the contact structure
\[
\xi_n=\ker(\sin(2\pi nz)\, dx + \cos(2\pi nz)\, dy)
\]
with $T^2\times\{pt\}$ isotopic to $T.$  The {\em Giroux torsion} of $\xi$ is the maximum of the Giroux torsion over all 
embedded tori in $M.$  (It need not be finite.)

We end this preliminary section with a simple lemma concerning curves on a surface. 
\begin{lem}\label{det-curves}
Let $\Gamma$ and $\Gamma'$ be two properly embedded multi-curves without homotopically trivial components on a
surface $\Sigma\not=S^2$. If $\gamma\cdot \Gamma\leq \gamma\cdot \Gamma'$ for all simple
closed curves and properly embedded arcs $\gamma$ in $\Sigma$ then, after isotopy, $\Gamma\subset \Gamma'.$
(Here $\cdot$ refers to geometric intersection number.)
\end{lem}
\begin{proof}
This proof, in the case where $\partial \Sigma=\emptyset$, originally appeared in \cite[Lemma 4.3]{Honda00b}
where it was attributed to W.~Kazez. We begin by representing $\Gamma$ as a set of non-isotopic
simple closed curves $\gamma_i$ with multiplicities $m_i,$ so $\Gamma$ is the union of $m_i$
disjoint copies of the $\gamma_i$. We can similarly represent $\Gamma'$ as curves $\gamma_i'$ with
multiplicities $m_i'.$ If $(\cup \gamma_i)\cdot (\cup \gamma_j')\not = \emptyset$ then there is some
$i$ and $j$ such that $\gamma_i\cdot \gamma'_j\not=\emptyset.$ In this case take $\gamma=\gamma'_j$
to contradict the hypothesis of the lemma.
Thus we are left to consider the case when the $\gamma_i$ and the $\gamma_j'$ are disjoint. In
this case, one can extend $(\cup\gamma_i)\cup (\cup\gamma'_j),$ where parallel curves are identified,  
to a set of curves that cut $\Sigma$
into a collection of pairs-of-pants (that is unless $\Sigma$ is a torus or annulus, in which case the lemma is
clear). It is now easy to verify that intersections with various curves
and properly embedded arcs determine the multiplicities on the curves associated to a pair-of-pants
decomposition.
\end{proof}

\begin{rem}\label{rem:curves}
A careful analysis of how to determine the multiplicities on the curves in a pair-of-pants
decomposition of a surface shows that if one only uses simple closed curves then the multi-curve
would be determined except for its multiplicities on the curves parallel to 
a boundary component of $\Sigma.$
\end{rem}

%%%%%%%%%%%%%%%%%%%%%%%%%%%%%%%%%%%%%%%%%%%%%%%%%%%%%%%%
\section{Fibered links and contact structures.}
%%%%%%%%%%%%%%%%%%%%%%%%%%%%%%%%%%%%%%%%%%%%%%%%%%%%%%%%
The goal of this section is to prove Theorem \ref{thm:G-compat}.  In essence, we determine all possible tight complements of a maximal self-linking number fibered transverse link $L.$  Indeed, we show they all arise from the contact structure $\xi_{(L,\Sigma)}$ supported by $L$ by adding Giroux torsion.  We begin with some conventions.

Let $L$ be a transverse, fibered link and $\Sigma$ be a convex fiber surface for $L$ with dividing set $\Gamma$ disjoint from $\partial \Sigma.$  We assume that the characteristic foliation $\mathcal{F}$ of $\Sigma$ is Morse-Smale and has no leaf connecting a negative singularity to the boundary of $\Sigma.$ % that the dividing set $\Gamma$ is disjoint from the boundary. %and contains curves isotopic to each boundary component and that $\partial \Sigma$ is contained in $\Sigma_-$.  
%We write the dividing set of $\Sigma$ as $\Gamma \cup \mathcal{C}$ where $\mathcal{C}$ is as set of components of the dividing set closest to the boundary and $\Gamma$ contains everything else.  Let $A$ be a union of annuli in $\Sigma$ that contain $\mathcal{C}$ and $\partial \Sigma$ so that $\Sigma'=\Sigma\setminus A$ is also a convex surface with transverse boundary and dividing set $\Gamma.$ 
Since $\Sigma$ is convex, there is an $I$-invariant neighborhood $\Sigma \times [0,1]$ of $\Sigma$.  After rounding corners, we can further assume this neighborhood to have a convex boundary $D\Sigma$, the double of $\Sigma$,  whose dividing 
set is given by $\Gamma \cup \bar{\Gamma} \cup \mathcal{C}$.  Here, $\Gamma$ is the dividing set of $\Sigma$ 
as it sits on $\Sigma \times \{1\},$ $\bar{\Gamma}$ is the dividing set of $\Sigma$ sitting on $\Sigma\times \{0\}$, and $\mathcal{C}$ a multi-curve isotopic to $L$ sitting in the region $(\partial\Sigma)\times [0,1]$ as the curves $(\partial \Sigma)\times\{\frac 12\}.$ One may easily see this by considering the proof of Lemma~\ref{dividing}.  (In particular, the convexity and dividing curves on $D\Sigma$ are deduced from the characteristic foliation.)
We may think of $L$ as being any transverse curve in $(\partial\Sigma)\times [0,1]$ (as they are all transversely isotopic). 
In particular, we may assume $L$ is $\mathcal{C}$, the boundary of $\Sigma\times\{1\}$, or $\Sigma\times\{0\}$ and will use whichever realization of $L$ is most convenient for our arguments without further comment.

Since $D \Sigma$ is a Heegaard surface for $M$ it divides $M$ into two handlebodies $H_1$ and $H_2.$  Each handlebody is diffeomorphic to $\Sigma \times I$, and we may choose the handlebodies such that the contact structure is $I$-invariant on $H_2$.  On the boundary of $H_2$, the dividing set on $D \Sigma$ is as above.  Thought of as on the boundary of $H_1$, the dividing set is given by $\Gamma \cup \mathcal{C} \cup \phi^{-1}(\bar{\Gamma}),$ where $\phi$ is the monodromy of the fibered link.  %Equivalently, this is a tight $\Sigma \times I$  with dividing set $\Gamma \cup \mathcal{C}$ on $\Sigma \times \{1\}$ and and $\phi(\Gamma) \cup \mathcal{C}$ on $\Sigma \times \{0\}$, having identified the two sets of components of $\mathcal{C}$ parallel to the binding.  

We point out that if there is a bypass for $D\Sigma$ with one endpoint on $\mathcal{C}$ and contained entirely in one of $D \Sigma - \Sigma \times \{1\}$ or $D \Sigma - \Sigma \times \{0\}$, then we can push $\Sigma$ across this bypass while keeping $L$ fixed.

\begin{Def}  Let $\xi$ be an oriented contact structure on a closed, oriented manifold $M$ and $(L, \Sigma)$ an open book for $M$.  We say $\xi$ and $(L, \Sigma)$ are {\em quasi-compatible} if there exists a contact vector field for $\xi$ which is everywhere positively transverse to the fibers of the fibration $(M\setminus L)\to S^1$ and positively tangent to $L$.
\end{Def}

This definition is very close to that of compatibility, save that we have eliminated the requirement that the vector field be Reeb.  It is this loosening that allows for twisting along $\Sigma$ and gives the following lemma.  We first need some conventions for \emph{adding half-twists.}  Let $T$ be an embedded, pre-Lagrangian torus in a contact manifold $(M, \xi)$.  Identify $T$ with $\R^2 / \Z^2 $ so that the the foliation on $T$ has slope 0.  To add $n$ half-twists along $T$, we first cut $M$ along $T,$ then glue in the contact manifold $(T^2 \times [0, n], \ker (\sin \pi t \, dx + \cos \pi t \, dy))$.  Note that if $T$ is a non-separating torus and $n$ is odd, then this will result in a non-orientable plane field.  Adding two half-twists is equivalent to adding a full-twist, that is, adding Giroux torsion.

\begin{lem} \label{lem:quasi-compatible}
Suppose $(M,\xi)$ is a contact manifold and $(L,\Sigma)$ is a fibered link, positively transverse to $\xi$. Assume that $\Sigma$ can be made convex with dividing set $\Gamma$ disjoint from $\partial \Sigma$ and having no component bounding a disk.  Moreover, assume that the characteristic foliation $\mathcal{F}$ of $\Sigma$ is Morse-Smale and has no leaf connecting a negative singularity to the boundary of $\Sigma.$  Let $H_1$ and $H_2$ be the two handlebody components of $M\setminus D\Sigma$, each of which is diffeomorphic to $\Sigma\times [0,1].$ % \cup \mathcal{C},$ where $\mathcal{C}$ is a copy of $L$ pushed into $\Sigma,$ and that no component of $\Gamma$ bounds a disk.  Let $A$ be the union of annuli in $\Sigma$ that contain $\mathcal{C}$ and $\partial \Sigma=L$ such that the characteristic foliation on $A$ is by arcs running from one boundary component to the other. Thus $\Sigma'=\Sigma\setminus A$ is a convex surface with transverse boundary and dividing set $\Gamma$ and $L$ and $\partial \Sigma'$ are transversely isotopic.
The following are equivalent:
\begin{itemize}
\item $(L,\Sigma)$ is quasi-compatible with $\xi$. 
\item $\xi$ restricted to $H_i$ is an $I$-invariant contact structure $i=1,2$.
\item $(M, \xi)$ cut along $\Sigma$ is an $I$-invariant contact structure.
\item There is a set of arcs $\alpha_1,\ldots, \alpha_n$ in $\Sigma$, cutting $\Sigma$ into a disk, such that the corresponding convex product disks in both $H_1$ and $H_2$ have dividing curves consisting of parallel arcs running from $\Sigma\times\{0\}$ to $\Sigma\times\{1\}.$
\end{itemize}

If $(L,\Sigma)$ is quasi-compatible with $\xi$, then the monodromy $\phi$ of $\Sigma$ satisfies $\phi(\Gamma) = \Gamma$.  Let $\gamma_1,\ldots, \gamma_k$ be disjoint non-homotopic curves on $\Sigma$ and $n_1,\ldots, n_k$ be positive integers such that $\Gamma$ is the union of $n_i$ copies of $\gamma_i, i=1,\ldots, k.$   The $k$ curves $\gamma_i$ will fall into some number of orbit classes, $\mathcal{O}_j=\{\gamma_{1_j},\ldots, \gamma_{k_j}\},$ under $\phi.$ Writing $M\setminus L$ as $\Sigma\times[0,1]/\sim_\phi,$ where $(x,1)\sim_\phi (\phi(x),0),$ the $\gamma_i\times[0,1]$ in an orbit class $\mathcal{O}_j$ give an incompressible torus $T_j$ in $M\setminus L.$  Then $\xi$ is obtained from $\xi_{(L,\Sigma)}$ by adding $n_i$ half-twists along the torus $T_i$ corresponding to $\mathcal{O}_i$.  
\end{lem}

\begin{proof}  
Once one observes the dividing set on $D\Sigma$ is as described above, the only non-obvious statement in the first paragraph is a slight generalization of the proof by Torisu.  Since $\xi$ is given by gluing together two $I$-invariant neighborhoods, it must be tight on $M\setminus L$ and uniquely determined by $\Gamma \cup \mathcal{C}$.  By building a model quasi-compatible contact structure, just as in the Thurston-Winkelnkemper construction but using a fiber with dividing set $\Gamma$, %\cup \mathcal{C}$ 
the proof is complete.  We construct a model contact structure, quasi-compatible with $(L, \Sigma)$ and having the desired dividing set on $\Sigma$ as follows.  Begin with a contact structure $\xi$ compatible with $(L, \Sigma)$.  For suitable choices in the Thurston-Winkelnkemper construction, the tori $T_i$ are pre-Lagrangian.  The proof is finished by observing that the Reeb vector field on $\xi_{(L,\Sigma)}$ that exhibits compatibility remains a contact vector field after adding half-twists along tori that are transverse to the open book $(L,\Sigma)$.   Thus, after adding half-torsion of order $n_i$ along tori $T_i$, the surface $\Sigma$ is convex and has dividing set $\Gamma \cup \mathcal{C}$.  This construction also clearly proves the second part of the lemma.
%As a final remark, notice that while we only get a single half-twist along each component of $\Gamma$, the fact that $\Gamma$ is a dividing set ensures that the resulting plane field is again orientable.
\end{proof}

\begin{proof}[Proof of Theorem \ref{thm:G-compat}]
In set-up, we assume that $(L,\Sigma)$ is a fibered link in $(M,\xi)$, that $L$ is transverse to $\xi$ with $sl_\xi{(L,\Sigma)} = -\chi(\Sigma)$, and that $\xi|_{M\setminus L}$ is tight.
We begin by isotoping  $\Sigma$ so
that its characteristic foliation is Morse-Smale. As mentioned above, working with convex surfaces
with transverse boundary can be tricky, so before we construct the contact vector field
transverse to $\Sigma$, we first adjust the characteristic foliation on $\Sigma$ so that $\partial 
\Sigma$ is ``protected'' from dividing curves. Specifically, we prove the following.
\begin{lem}
All negative hyperbolic singularities in the characteristic foliation of $\Sigma$ can be cancelled by
a $C^0$-small isotopy supported away from $\partial \Sigma.$
\end{lem}
%One might wish to skip the proof of this lemma for the moment so as to see why we are interested in canceling these singularities. 
\begin{proof}
We begin by introducing positive elliptic/hyperbolic pairs of singularities along positive periodic 
orbits in the characteristic foliation of $\Sigma$ so we can assume there are no such periodic
orbits.  
Notice that since $sl(K)=-\chi(\Sigma)$
we can use Formula~\ref{slformula} to conclude that the number of negative elliptic points in
the characteristic foliation of $\Sigma$ is equal to the number of negative hyperbolic points: 
$e_-=h_-.$ If $p$ is a negative elliptic point, then let $B_p$ be the closure of all points in $\Sigma$
that lie on a leaf that limits to $p.$ This is called the {\em basin} of $p.$ The interior of the basin is
clearly an embedded open disk in $\Sigma.$ Moreover, $B_p$ is an immersed disk in $\Sigma$ possibly 
with double points along its boundary, see \cite{Eliashberg92a, Etnyre03}. If this basin is an embedded 
disk, then it is easy to see its boundary is made up of positive elliptic points and the stable manifolds
of hyperbolic points. If all the hyperbolic points are positive then we can cancel them with the
elliptic points to create an overtwisted disk. Thus there must be a negative elliptic point on 
the boundary of the basin with which we can cancel $p.$ If $B_p$ is not embedded then there must
still be a negative hyperbolic point on $\partial B_p.$ See \cite{Eliashberg92a, Etnyre03}. Thus we
can always eliminate $p.$ In particular, we can eliminate all the negative elliptic points without
moving $\partial \Sigma.$ Thus, since we must have $e_-=h_-,$ we have also eliminated all the 
negative hyperbolic points as well. 
\end{proof}

Assuming there are no negative hyperbolic singularities, we construct a contact vector field
transverse to $\Sigma$ as in the proof of Lemma~\ref{dividing} sketched above. We see that there 
will be a dividing curve parallel to each boundary component and that the annulus it cobounds with the boundary 
has characteristic foliation consisting of arcs running from one boundary component to the other. Denote these dividing curves by ${C}$, and let
$\Gamma$ denote the remainder of the dividing set on $\Sigma.$   %considered in the first part of this section. 
Thus by Theorem~\ref{realize}, we can realize any characteristic foliation on $\Sigma$ that is 
divided by our given dividing curves and is the given foliation on the annuli cobounded by ${C}$ and $\partial \Sigma.$  Since $\xi$ is 
assumed to be tight this implies that no dividing curve on $\Sigma$ bounds a disk. Thus the 
components of $\Sigma_\pm$ have non-positive Euler characteristic.  From 
Lemma~\ref{sl-formula} we conclude that $\chi(\Sigma_-)=0$ and that $\Sigma_-$ is a union of 
annuli. (Notice that if we did not have the dividing curves parallel to the boundary, we could not use the Legendrian realization principle and 
conclude that the components of $\Sigma_\pm$ had non-positive Euler characteristic. Thus we 
say these dividing curves ``protect'' the boundary.) Now we can transversely isotop $L$ past $C$ to create a new convex surface, with boundary $L$, having dividing set $\Gamma$ and characteristic foliation as considered in the first part of this section.

From the proof of Lemma~\ref{lem:quasi-compatible} and the set-up above, it is clear that Theorem~\ref{thm:G-compat} is equivalent to showing that $L$ has maximal self-linking number if and only if $L$ is quasi-compatible with $\xi$.  In order to prove $L$ is quasi-compatible with $\xi$, we need two lemmas which are proved below.
\begin{lem} \label{lm:Reduce} If $\phi$ is the monodromy map of the fibration of $M\setminus L$ and  $\phi(\Gamma)$ is not isotopic to $\Gamma$, then we may find a new convex fiber surface $\Sigma'$ for $L$, isotopic to $\Sigma$, with one of the following being true:
\begin{itemize}
\item $\Sigma'$ has fewer dividing curves than $\Sigma$, or
\item $\Sigma'$ has the same number of dividing curves but fewer null-homologous components.
\end{itemize}
\end{lem}
\begin{lem} \label{lm:Remove} If $\phi(\Gamma)$ is isotopic to $\Gamma$, then either $\xi$ is quasi-compatible with $(L, \Sigma)$ or we may find a bypass for $\Sigma$ that either reduces the number of dividing curves or reduces the number of null-homologous components of $\Gamma$.
\end{lem}

%Now, as in the proof of Torisu's Lemma~\ref{support} notice that $\xi$ is quasi-compatible with $\xi_{(L,\Sigma)}$ if and only if there is a convex disk decomposition of $M \setminus D \Sigma$ where the dividing set of each disk is a set of parallel arcs that run from $\Sigma \times \{0\}$ to $\Sigma \times \{1\}$.  
We apply Lemmas \ref{lm:Reduce} and \ref{lm:Remove} to simplify the dividing set on $\Sigma$, noting that at every stage we either reduce the number of dividing curves or the number of null-homologous components, until we are left with either an $(L,\Sigma)$ which is quasi-compatible with $\xi$ or a convex surface $\Sigma$ with $\Gamma=\emptyset$ which, by Torisu's Lemma~\ref{support}, implies that $(L,\Sigma)$ is compatible with $\xi.$ The theorem now clearly follows from Lemma~\ref{lem:quasi-compatible}.
%an arc decomposition with no bypasses for $\Sigma$ on it.  This can happen by either reducing the dividing set of $\Sigma$ to $\mathcal{C},$ that is $\Gamma=\emptyset,$ or $\Gamma\not=\emptyset$ and  we have the above mentioned convex disk decomposition of $M\setminus D\Sigma.$ Thus $\xi$ is either compatible or quasi-compatible with $(L,\Sigma).$ %Such an arc decomposition shows that $\xi$ is composed of two $I$-invariant neighborhoods which glue and exhibit the quasi-compatiblility of $\xi$ and $(L, \Sigma)$.
\end{proof}

\begin{proof}[Proof of Lemma~\ref{lm:Reduce}]  Since $\phi(\Gamma)$ is not isotopic to $\Gamma$, we apply Lemma \ref{det-curves} and Remark~\ref{rem:curves} to find a simple closed curve  $\alpha$ on $\Sigma$ that hits $\Gamma$ a different number of times than $\phi(\Gamma)$.  Legendrian realize $\alpha$ on $\Sigma \times \{\pm 1\}$ and look at a convex annulus traced out by $\alpha$.  (Note, if $\alpha$ doesn't intersect one of $\Gamma$ or $\phi(\Gamma)$ we may need to 
use the super Legendrian realization trick where one first realizes a Legendrian curve elsewhere on the surface, then creates two dividing curves parallel to it by folding the surface.  See \cite{Honda00b}. If that is the case, the bypass we find will be on the other copy of $\Sigma$ and so folding will not affect the number or pattern of dividing curves on our fiber surface.)  By the Imbalance Principle \cite[Proposition 3.17]{Honda00a}, there exists a bypass for $D\Sigma$ along $(\Sigma \times\{ i \})$, for either $i=0$, or $i= 1$.  We will show that we can now either reduce the number of dividing curves or reduce the number of null-homologous components of $\Gamma$. Call the bypass arc $B$.

The case we would like to reduce to is when the bypass $B$ involves three different curves.  In this case, pushing $\Sigma$ across this bypass will reduce the number of dividing curves, thus finishing the proof.

When $B$ involves only two different curves (note by choosing the original curve $\alpha$ to intersect $\Gamma$ minimally, $B$ must involve either two or three dividing curves), we use Lemma~\ref{moveleft} to find a bypass that either reduces $|\Gamma|$, or preserves $|\Gamma|$ but decreases the number of null-homologous components of $\Gamma$.  Since the two curves involved must be parallel (recall that $\Sigma_-$ consists entirely of annuli), they cobound an annulus $A$.  Label the isotopy class of the curves by $b_1$.  Then $A$ together with a small neighborhood of $B$ is either a punctured torus or a pair of pants.  

{\bf Case 1:} $A \cup \nu B$ is a torus.  Applying Lemma \ref{moveleft}, we may find a bypass along any arc that intersects $A$ as $B$ does and then has one of its end points on some other dividing curve. Thus we may find a different bypass involving a third component of $\Gamma \cup \mathcal{C}$ elsewhere in $\Sigma$.  Pushing $\Sigma$ across this bypass reduces the number of dividing curves.

For the remaining cases, $A \cup \nu B$ is a pair of pants.  As long as $b_1$ is homologically essential in $\Sigma$, we can find a bypass involving a third component of the dividing set.  If not, we can only ensure that there is a bypass so that after isotoping $\Sigma$ across it, the two dividing curves are non-separating.  

{\bf Case 2:} $b_1$ is homologically essential.  The pair of pants $P$ has three boundary components: labelled $b_1$, $b_2$ and $b_3$ as follows:  the dividing curves on $A$ are parallel to $b_1$, and become parallel to $b_2$ after pushing $\Sigma$ across the bypass.  The remaining boundary component is $b_3$.  By Lemma \ref{moveleft}, for any arc $B'$ beginning on $b_1,$ exiting $b_3$ and ending on another dividing curve, there is also a bypass for $\Sigma$ along $B'$.  There are three possibilities: (a) $b_3$ is adjacent to the boundary of $\Sigma$, (b) $b_3$ and $b_1$ cobound a subsurface $C_{13}$ or (c) $b_3$ bounds a surface $C_3$.  In the first case, we may find a bypass exiting $b_3$ and involving a third curve, thus reducing $|\Gamma|$. 

In case (b), Lemma \ref{moveleft} gives us a new bypass starting on $b_1$, exiting $b_3$ and possibly returning to $b_1$.  This bypass either involves three dividing curves or $A \cup \nu B'$ is a punctured torus and we are in Case 1.  
If $b_3$ bounds a surface $C_3$, we are in case (c) and we see that $C_3$ cannot be a disk by our choice of $\alpha$ above. Thus we can find a new bypass along an arc $B'$ that starts on $b_1$, intersects $C_3$ in a non-separating curve and then returns to $b_1.$   With respect to this bypass, $b_3'$ is now adjacent to $b_2'$ and hence some component of the boundary of $\Sigma$, as in case (a) above.

%{\bf Case 3:} $b_1$ is homologically essential and $b_3$ is not. We note that the surface $C_3$ that $b_3$ bounds must have genus by our choice of $\alpha$ above. Moreover, the surface $C_1$ adjacent to $b_1$ in $\Sigma\setminus A \cup \nu B$ contains topology as well. Thus we can find a curve $\gamma$ that intersects $C_3$ and $C_1$ in a non separating arc and intersects $A \cup \nu B$ in two arcs running from $b_1$ to $b_3.$ Let $\Sigma'$ be the surface $\Sigma$ after it is pushed across the bypass along $B.$ We can assume that $\Sigma$ and $\Sigma'$ cobound a region $\Sigma\times I.$ The curve $\gamma$ gives an annulus in $\Sigma\times I$ which intersects the dividing set on $\Sigma$ in four points and the dividing set on $\Sigma'$ in no points. We can Legendrian realize the boundary of the annulus, make it convex and thus find a bypass for $\Sigma$ along a subarc of $\gamma.$ Since $\gamma$ intersects $C_1$ and $C_3$ in a non separating arcs one sees that this new bypass will be of the type considered in Case 2.

{\bf Case 3:} $b_1$ bounds a subsurface of $\Sigma$.  If $b_2$ is homologically essential, then pushing $\Sigma$ across $B$
will reduce the number of null-homologous components of $\Gamma$.  If not, then $b_3$ is either adjacent to $\partial \Sigma$ or it also separates a subsurface $C_3$ off of $\Sigma.$ %(possibly containing $\partial \Sigma$).  
The first case can be handle as in Case 2(a) above. The second case gives a configuration similar to Case 2(c) above and allows us to either reduce $|\Gamma|$ or reduce to the case when $b_2$ is homologically essential. %Find an nonseparating, properly embedded arc in $C_3$ and extend it to a bypass arc $B'$ as in Lemma \ref{moveleft}.  With respect to the new arc, neither $b_2$ nor $b_3$ is separating and so pushing $\Sigma$ across $B'$ will decrease the number of null-homologous components of $\Sigma$.
\end{proof}

\begin{proof}[Proof of Lemma~\ref{lm:Remove}] First observe that $\xi$ is quasi-compatible with $\xi_{(L,\Sigma)}$ if and only if there is a convex disk decomposition of $M \setminus D \Sigma$ where the dividing set of each disk is a set of parallel arcs that run from $\Sigma \times \{0\}$ to $\Sigma \times \{1\}$.  Assume then, that $\xi$ is not quasi-compatible with $(L, \Sigma)$ and hence there is some bypass for $\Sigma$.  Just as in Lemma \ref{lm:Reduce}, given a bypass, we can reduce the number of dividing curves except possibly when it involves only a pair of null-homologous curves.  Again, in this case, we can find a bypass that reduces the number of null-homologous components.  
\end{proof}

\section{Classification of Bennequin bound realizing fibered links}
The goal of this section is to prove Theorem~\ref{knotthm} and Corollary~\ref{classifymax}. %and \ref{thm:OT}.

\begin{proof}[Proof of Theorem~\ref{knotthm}]
We begin by proving the second statement in the theorem. 
To this end let $L_1$ and $L_2$ be two transverse links in $(M,\xi)$ realizing the link type $L$ and satisfying $sl_\xi(L_1, \Sigma) = sl_\xi(L_2, \Sigma) = -\chi(\Sigma)$.  By Theorem \ref{thm:G-compat}, $\xi$ is quasi-compatible with both $L_1$ and $L_2$.  Since $L_i$ has maximal self-linking number, the dividing set on $\Sigma$ must consist of a single copy of the boundary in addition to some number of pairs of parallel curves.  By assumption, however, $\xi$ has trivial Giroux torsion and so the dividing set on any quasi-compatible convex fiber surface must be a single copy of the boundary.  Equivalently, $\xi$ is compatible (in the strict sense) with both open books $(L_i,\Sigma)$ $i=1,2.$  Since the open book decompositions are isotopic, there exists an isotopy of contact structures taking $L_1$ to $L_2$, which, by Gray's theorem, induces a contactomorphism of $\xi$ that takes $L_1$ to $L_2.$ 
%The open books are isotopic which induces a contact isotopy from $(\xi, L_1)$ to $(\xi, L_2)$.  Normalizing near $L$, we can think of this as a contact isotopy from $\xi|{ M-L_1}$ to $\xi|{ M-L_2}$.  

If $(M,\xi)$ was the tight contact structure on the sphere then any such isotopy can be made into a isotopy that preserves the contact structure, directly implying that $L_1$ and $L_2$ are transversely isotopic. See \cite{Etnyre05}.
%\Delta(h)$, for some .
%From the proof of Theorem~\ref{main} we see that each one bounds a convex Seifert surface $\Sigma_1$ and $\Sigma_2$ with dividing set a single %copy of $L_1$ and $L_2,$ respectively. (We can also take  $\Sigma_1$ diffeomorphic to $\Sigma_2.$) We easily see that $\Sigma_1$ has a %neighborhood $N_1$ that is contactomorphic to a neighborhood $N_2$ of $\Sigma_2.$ We may, moreover, assume that $\partial N_i$ is convex with %a single dividing curve, that $K_i$ is on $\partial N_i$ and parallel to the dividing curve, and the contactomorphism takes $K_1$ to $K_2.$ %Arguing as in Lemma~\ref{support} we see that the contactomorphism from $N_1$ to $N_2$ may be extended to a contactomorphism of $M$ to$M$ that %takes $K_1$ to $K_2.$
\end{proof}

\begin{proof}[Proof of Corollary~\ref{classifymax}]
It is well-known that the closure of a positive braid is fibered \cite{Stallings78}. In addition Bennequin~\cite{Bennequin83} showed that any closed braid naturally gives a transverse knot with self-linking equal to minus the braid index plus the algebraic length of the word representing the braid.  The result now follows from Theorem~\ref{knotthm}. Bennequin's formula for the self-linking number of a closed braid shows that a strongly quasi-positive knot has a transverse representative with self-linking number equal to minus the Euler characteristic of its quasi-positive surface. 
\end{proof}
%\begin{proof}[Proof of Theorem \ref{thm:OT}]

%The proof rests on two simple observations.  First, we note that adding Giroux torsion doesn't change the homotopy type of the plane field.  Second, observe that the exteriors constructed by $G$-compatibility with different dividing sets are all distinct.  So take any fibered link $(L, \Sigma)$ in $M$ and let $\xi$ be a compatible contact structure.  Then the transverse links $L_i$ obtained by doing a full Lutz twist of order $i$ along $L$ all have maximal self-linking number and reside in overtwisted contact structures in the same homotopy class, that is, in the same contact structure.  Thus the $L_i$ represent infinitely many distinct transverse links in $\xi'$ with the same self-linking number.
%\end{proof}

%This is also a possible way to generate distinct transverse links in a tight contact manifold.  Whenever $\xi$ is obtained by Giroux torsion along tori disjoint from the binding of the open book \emph{in distinct ways}, we realize different exteriors for maximal transverse links in the same contact manifol.

\section{Braid representatives of links}
The goal of this section is to prove the results~\ref{qp-rep} through~\ref{minrep}. To this end we begin by recalling that the closure of any braid represents a transverse knot in the standard tight contact structure on $S^3.$ In addition we know the following.
\begin{thm}[Orevkov Shevchishin 2003, \cite{OrevkovShevchishin03}; Wrinkle 2002, \cite{Wrinkle02}]\label{transmarkov}
Two braids representing the same transverse link are related by positive Markov moves and conjugation in the braid group. 
\end{thm}

We also need the following result.
\begin{thm}[Orevkov 2000, \cite{Orevkov00}]\label{qpifstab}
A braid $B$ is quasi-positive if and only if any positive stabilization is quasi-positive. 
\end{thm}

\begin{proof}[Proof of Theorem~\ref{qp-rep}]
Let $B$ represent a fibered strongly quasi-positive link $L.$  If $a(B)=n(B)-\chi(L)$ then the transverse link $T_B$ given as the closure of $B$ has self-linking number equal to $-\chi(L),$ so by Corollary~\ref{classifymax}, $T_B$ is transversely isotopic to the closure of a strongly quasi-positive braid. Therefore $B$ becomes strongly quasi-positive by a sequence of positive Markov moves (and braid conjugations) and hence $B$ is quasi-positive by Theorem~\ref{qpifstab}.

For the other implication assume that $B$ is a quasi-positive braid and let $T_B$ be its closure. We know that $sl(T_B)=n(B)-a(B)$ and the algebraic length $a(B)$ of $B$ is just the number of generators $\sigma_{i,j}$ used to represent $B.$ Suppose that 
\[
a(B)< n(B)-\chi(L_B).
\]
Recall that the link $L_B$ that $B$ represents is also the closure of a strongly quasi-positive braid $B'.$ We know that we can construct a Seifert surface for $L$ of genus $g'=\frac 12(1-n(B')+a(B'))$ and that this surface is of minimal possible genus. It is also the minimal slice genus surface (see\cite{Rudolph93}). Since $L$ supports the standard tight contact structure on $S^3$ we know that 
\[a(B')=n(B')-\chi(L_B).\]
By positively (braid) stabilizing either $B$ or $B'$ as necessary we can assume that $n(B)=n(B')$ so that we see $a(B)<a(B').$

By writing $B$ as the product of $a(B)$ conjugates of positive generators in the braid group it is easy to construct a ribbon immersed surface in $S^3$ with boundary $L_B$ (see, for example, \cite{Rudolph05}) and hence a slice surface in $B^4,$ of genus $g=\frac 12(1-n(B)+a(B)).$ Notice that $g<g'$ contradicting the minimality of the slice surface for the closure of $B'.$ Thus we must have $a(B)=n(B)-\chi(L_B).$
\end{proof}

We are now ready to prove the relation between quasi-positive representatives of strongly quasi-positive links.
\begin{proof}[Proof of Theorem~\ref{stab}]
Let $B_1$ and $B_2$ be two quasi-positive braids whose closure is $L,$ a fibered link that is also the closure of a strongly quasi-positive braid $B.$ 
From Theorem~\ref{qp-rep} we know that the transverse links $T_{B_1},T_{B_2}$ and $T_B$ obtained as the closures of these braids all have maximal self-linking number and hence by Corollary~\ref{classifymax} are transversely isotopic. Thus Theorem~\ref{transmarkov} finishes our proof.
%Suppose $B_l$ is an $n_l$-braid that is a product of $k_l$ generators $\sigma_{i,j}.$ It is easy to see $B_l$ bounds a Seifert surface with Euler characteristic $\chi_l=n_l-k_l$ and moreover $B_l$ represents a transverse link with total self-linking number $sl_l=-n_l+k_l.$ Since $sl_l=-\chi_l$ and we know that if $\Sigma$ is the fiber of the fibration of $S^3\setminus L$ then $sl_l\leq -\chi(\Sigma)\leq -\chi_l.$ Thus $sl_1=sl_2=-\chi(\Sigma).$ That is the transverse links associated to $B_1$ and $B_2$ both have maximal self-linking number. Thus by Theorem~\ref{classifymax} we know that the transverse links associated to $B_1$ and $B_2$ are transversely isotopic and so Theorem~\ref{transmarkov} finishes our proof. 
\end{proof}

Lastly we prove our result about minimal braid representatives.
\begin{proof}[Proof of Theorem~\ref{minrep}]
Let $B$ be a strongly quasi-positive braid and let $K_B$ be its closure. If the knot type $K_B$ satisfies the Braid Geography Conjecture then let $B'$ be a minimal braid index braid representing $K_B.$ Set $n=n(B')$ and $w=a(B').$  We know from the conjecture that $n(B)= n+x+y$ and $a(B)=w+x-y$ for some non-negative integers $x$ and $y.$ Moreover,
\[sl(K_{B'})=-n(B)+a(B)= -n+w-2y=sl(K_B)-2y\]
and since $sl(K_{B'})$ is the maximal possible self-linking number for transverse knots in the knot type $K_B$ we see that $sl(K_B)=sl(K_{B'}).$ Hence $K_B$ and $K_{B'}$ are transversally isotopic by Corollary~\ref{classifymax} and related by positive Markov moves by Theorem~\ref{transmarkov}. We can finally conclude, using Theorem~\ref{qpifstab}, that $B'$ is quasi-positive. 
\end{proof}

\section{Knots classifying contact structures}
In this section we prove Theorem~\ref{classifies}. We begin by recalling the following theorem.
\begin{thm}[Colin and Honda 2008, \cite{ColinHonda08}]
On a closed oriented 3-manifold $M,$ every oriented, positive contact structure is supported by an open book whose binding is connected, and whose monodromy is freely homotopic to a pseudo-Anosov homeomorphism. \qed
\end{thm}

\begin{proof}[Proof of Theorem~\ref{classifies}]
It is clear that if two contact structures $\xi_1$ and $\xi_2$ are isotopic then $|\mathcal{T}_n^{\xi_1}({K})|=|\mathcal{T}_n^{\xi_2}({K})|$ for all $K\in \mathcal{C}_\text{fpa}$ and $n\in \Z.$ Thus we are left to prove the other implication. 

We first note that from the theorem just stated, every contact structure on $M$ is supported by an open book with binding in $\mathcal{C}_\text{fpa}.$  We also notice that given any knot type ${K}$ in $M,$ a contact structure $\xi$ on $M$ is overtwisted if and only if $|\mathcal{T}^\xi_n(K)|>0$ for all $n\in \Z.$ Thus we can distinguish tight from overtwisted contact structures on $M$ using $\mathcal{C}_\text{fpa}.$

For the next two possibilities, let $\xi_1$ and $\xi_2$ be two contact structures which $|\mathcal{T}_n^{\xi_1}({K})|=|\mathcal{T}_n^{\xi_2}({K})|$ for all $K\in \mathcal{C}_\text{fpa}$ and $n\in \Z.$  We will show that $\xi_1$ and $\xi_2$ must be isotopic.  
 
For the first case, assume that $\xi_1$ and $\xi_2$ are both tight contact structures on $M.$ Let $B\in \mathcal{C}_\text{fpa}$ be the binding of an open book supporting $\xi_1.$ Then 
\[
|\mathcal{T}^{\xi_1}_{-\chi(B)}(B)|=1
\]
from Corollary~\ref{paclass} and the fact that transverse knots are classified up to contactomorphism by the contact structures on their complements. 
Hence we have
\[
|\mathcal{T}^{\xi_2}_{-\chi(B)}(B)|=1.
\]
But,  again using Corollary~\ref{paclass}, we see that $B$ supports $\xi_2$ and hence $\xi_1$ and $\xi_2$ are isotopic.

Finally assume that $\xi_1$ and $\xi_2$ are both overtwisted contact structures on $M.$ Let $B\in \mathcal{C}_\text{fpa}$ be the binding of an open book supporting $\xi_1.$ Thus 
 \[
|\mathcal{T}^{\xi_1}_{-\chi(B)}(B)|\geq 2
\]
since there is at least one transverse knot in $\mathcal{T}^{\xi_1}_{-\chi(B)}(B)$ with tight complement and another one with overtwisted complement. Hence 
  \[
|\mathcal{T}^{\xi_2}_{-\chi(B)}(B)|\geq 2.
\]
If all the transverse knots in $\mathcal{T}^{\xi_2}_{-\chi(B)}(B)$ have overtwisted complement then the next lemma implies $|\mathcal{T}^{\xi_2}_{-\chi(B)}(B)|=1.$ Thus there is a $T\in \mathcal{T}^{\xi_2}_{-\chi(B)}(B)$ with tight complement. Since the monodromy of $B$ is pseudo-Anosov the only incompressible tori in $M-N(B),$ where $N(B)$ is a neighborhood of $B,$  are parallel to $\partial (M-N(B)).$ Thus Theorem~\ref{thm:G-compat} says that either $B$ supports the contact structure $\xi_2$ or $\xi_2$ is obtained from the contact structure supported by $B,$ which of course is $\xi_1,$ by some number of full Lutz twists along $B$ (this is, of course, equivalent to adding Giroux torsion along a torus parallel to the boundary of $M-N(B)$).  However, since full Lutz twists do not change the homotopy type of a plane field Eliashberg's classification of overtwisted contact structures \cite{Eliashberg89} implies $\xi_1$ and $\xi_2$ are isotopic. 
\end{proof}
\begin{lem}
Let $T$ and $T'$ be topologically isotopic, null-homologous transverse knots in a contact manifold $(M,\xi).$ If $sl(T)=sl(T')$ then there are neighborhoods $N$ and $N'$ of $T$ and $T',$ respectively, such that $\xi$ restricted to $N$ is contactomorphic to $\xi$ restricted to $N'$ and $\xi|_{\overline{M-N}}$ is homotopic rel boundary as a plane field to $\xi|_{\overline{M-N'}}.$

Moreover, if $\xi|_{\overline{M-N}}$ and  $\xi|_{\overline{M-N'}}$ are both overtwisted then they are isotopic rel boundary as contact structures and there is a contactomorphism of $(M,\xi)$ taking $T$ to $T'$. 
\end{lem}
\begin{proof}
The first part of the lemma follows from a relative version of the Pontryagin - Thom construction and the fact that in the given situation the relative Euler class of $\xi$ restricted to $\overline{M-N}$ (or $\overline{M-N'}$) is determined by the Euler class of $\xi$ on $M$ and $sl(T)$ (or $sl(T')$). (One needs to be careful about the framed cobordism classification of framed links coming form the Pontryagin - Thom construction in the relative setting.) For details see \cite{EtnyrePre??}.

The second part of the lemma follows from Eliashberg's classification of overtwisted contact structures \cite{Eliashberg89} which holds in the relative setting. 
\end{proof}

\def\cprime{$'$} \def\cprime{$'$}

%\bibliography{references}
%\bibliographystyle{plain}

\end{document}